\documentclass [11pt]{article} 
\title{Many $T$ copies in $H$-free graphs} 
\author{ Noga Alon \thanks{Sackler School of Mathematics and Blavatnik
School of Computer Science, Tel Aviv University, Tel Aviv 69978, Israel
and School of Mathematics, Institute for Advanced Study, Princeton, NJ
08540. Email: nogaa@tau.ac.il. Research supported in part by a
USA-Israeli BSF grant, by an ISF grant, by the Israeli I-Core program
and by the Oswald Veblen Fund.} \and Clara Shikhelman \thanks{Department
of Mathematics, Tel Aviv University, Tel Aviv 69978, Israel. Email:
clarashk@mail.tau.ac.il. Research supported in part by an ISF grant. } }
\date{}
\usepackage{amsfonts}
\usepackage{amsmath} 
\usepackage{amsthm} 
\newtheorem{theo}{Theorem}[section] 
\newtheorem{prop}[theo]{Proposition}
\newtheorem{lemma}[theo]{Lemma}

\newcommand{\HH}{{\cal H}} 
\newcommand{\N}{{\mathcal{N}}}

\newtheorem{theorem}{Theorem}[section]
\newtheorem{proposition}[theorem]{Proposition}

\newtheorem{definition}[theorem]{Definition}
\newtheorem{remark}[theorem]{Remark}

\newtheorem{claim}[theorem]{Claim}

\def\E{\mathop{\mathbb E}} 
\newcommand{\F}{\mathbb F}
\newcommand{\HY}{\mathcal H}

\begin{document} 
\maketitle 

\begin{abstract} For two graphs $T$ and $H$
with no isolated vertices and for an integer $n$, let $ex(n,T,H)$ denote
the maximum possible number of copies of $T$ in an $H$-free graph on $n$
vertices. The study of this function when $T=K_2$ is a single edge is
the main subject of extremal graph theory. In the present paper we
investigate the general function, focusing on the cases of triangles,
complete graphs, complete bipartite graphs and trees. These cases reveal
several interesting phenomena. Three representative results are:\\ (i)
$ex(n,K_3,C_5) \leq (1+o(1)) \frac{\sqrt 3}{2} n^{3/2},$\\ (ii) For any
fixed $m$, $s \geq 2m-2$ and $t \geq (s-1)!+1 $,
$ex(n,K_m,K_{s,t})=\Theta(n^{m-\binom{m}{2}/s})$ and \\ (iii) For any
two trees $H$ and $T$, $ex(n,T,H) =\Theta (n^m)$ where $m=m(T,H)$
is an integer depending on $H$ and $T$ (its precise definition is given
in Section 1).

The first result improves (slightly) an estimate of Bollob\'as and
Gy\H{o}ri. The proofs combine combinatorial and probabilistic arguments
with simple spectral techniques. 
\end{abstract}

\section{Introduction}

For two graphs $T$ and $H$ with no isolated vertices and for an integer
$n$, let $ex(n,T,H)$ denote the maximum possible number of copies of $T$
in an $H$-free graph on $n$ vertices.

When $T=K_2$ is a single edge, $ex(n,T,H)$ is the well studied function,
usually denoted by $ex(n,H)$, specifying the maximum possible number of
edges in an $H$-free graph on $n$ vertices. There is a huge literature
investigating this function, starting with the theorems of Mantel
\cite{Ma} and Tur\'an \cite{Tu} that determine it for $H=K_r$. See, for
example, \cite{Si} for a survey.

In the present paper we show that the function  for other graphs $T$
besides $K_2$ exhibits several additional interesting features. We
illustrate these by presenting several general results and by focusing
on various special cases of graphs $H$ and $T$ in certain prescribed
families. 
The question is interesting for many other graphs $T$ and $H$, and many
of the results here can be extended.

There are several sporadic papers dealing with the function $ex(n,T,H)$
for $T \neq K_2$. The  first one is due to Erd\H{o}s in \cite{Er62},
where he determines $ex(n,K_t,K_r)$ for all $t <r$ (see also \cite{Bo}
for an extension). A notable recent example is given in \cite{HHKNR},
where the authors determine this function precisely for $T=C_5$ and
$H=K_3$. 

The case $T=K_3$ and $H=C_{2k+1}$ has also been studied. Bollob\'as and
Gy\H{o}ri \cite{BG} proved that 
\begin{equation} \label{e11} (1+o(1))
\frac{1}{3 \sqrt 3}n^{3/2} \leq ex(n,K_3,C_5) \leq (1+o(1))
\frac{5}{4}n^{3/2}. 
\end{equation}

Gy\H{o}ri and Li \cite{GL} proved that for any fixed $k \geq 2$
\begin{equation} \label{e12} 
\binom{k}{2} ex_{bip}
(\frac{2n}{k+1},C_4,C_6, \ldots ,C_{2k}) \leq ex(n,K_3,C_{2k+1}) \leq
\frac{(2k-1)(16k-2)}{3} ex(n,C_{2k}), 
 \end{equation}

\noindent where $ex_{bip} (m,C_4,C_6, \ldots ,C_{2k})$ denotes the
maximum possible number of edges in a bipartite graph on $m$ vertices
and girth exceeding $2k$.

Here we start with a simple characterization of the pairs of graphs $H$
and $T$ for which $ex(n,T,H)=\Theta(n^t)$, where $t$ is the number of
vertices of $T$. Combining this observation with the graph removal lemma
we establish an Erd\H{o}s-Stone type result by giving an asymptotic
formula for $ex(n,K_t,H)$ for any graph $H$ with chromatic number
$\chi(H)>t$.

Next we study the case $T=K_3$. Our first result characterizes all
graphs $H$ for which $ex(n,K_3,H) \leq c(H) n$. The friendship graph
$F_k$ is the graph consisting of $k$ triangles with a common vertex.
Equivalently, this is the graph obtained by joining a vertex to all $2k$
vertices of a matching of size $k$. Call a graph an extended friendship
graph iff its $2$-core is either empty or $F_k$ for some positive $k$.
\vspace{0.3cm}

\noindent \begin{theorem} \label{t11} There exists a constant $c(H)$ so
that $ex(n,K_3,H)\leq c(H)n$ if and only if $H$ is a subgraph of an
extended friendship graph. \end{theorem}

We also slightly improve the upper estimates in (\ref{e11}) and in
(\ref{e12}) above, proving the following. \begin{prop} \label{p13} The
following upper bounds hold.

\noindent (i) $ex(n,K_3,C_5) \leq (1+o(1)) \frac{\sqrt 3}{2} n^{3/2}.$

\noindent (ii) For any $k \geq 2$, $ex(n,K_3,C_{2k+1}) \leq
\frac{16(k-1)}{3} ~ ex(\lceil n/2 \rceil,C_{2k}).$ \end{prop}

A similar result has been proved independently by F\"uredi and \"Ozkahya
\cite{FO}, who showed that $ex(n,K_3,C_{2k+1}) \leq 9k ~ex (n, C_{2k})$.

The next theorem deals with maximizing the number of copies of a
complete graph while avoiding complete bipartite graphs: \vspace{0.3cm}

\noindent \begin{theorem} \label{thm:K_mK_s,t} For any fixed $m$ and $t
\geq s $ satisfying $s\geq 2m-2$ and $t \geq (s-1)!+1$ there are two
constants $c_1=c_1(s,t)$ and $c_2=c_2(s,t)$ such that $$ c_1
n^{m-\binom{m}{2}/s} \leq ex(n,K_m,K_{s,t}) \leq c_2
n^{m-\binom{m}{2}/s}. $$ \end{theorem}

The last two theorems focus on the case where the excluded graph $H$ is a
tree. Before stating  the results we give the following definitions:

\begin{definition} For a graph $T$, a set of vertices $U\subseteq V(T)$
and an integer $h$, the $(U,h)$ blow-up of $T$ is the following graph.
Fix the vertices in $U$, and replace each connected component in
$T\setminus U$ with $h$ vertex disjoint copies of it connected to the
vertices of $U$ exactly as the original component is connected to these
in $T$. \end{definition}

\begin{definition} For two trees, $T$ and $H$, let $m(T,H)$  be the
maximum integer $m$ such that there is a $(U,|V(H)|)$ blow-up of $T$
containing no copy of $H$ and having $m$ connected components in
$T\setminus U$. \end{definition}

In this notation we prove the following.

\begin{theorem} 
\label{thm:2trees} 
For every two integers $t$ and $h$ there are positive constants 
$c_1(t,h),c_2(t,h)$ so that the following holds.
Let $H$ be a tree on $h$ vertices and let $T$ be a tree on $t$
vertices, then
$$ 
c_1(t,h)n^m \leq ex(n,T,H) \leq c_2(t,h) n^m, 
$$ 
where $m=m(T,H)$. 
\end{theorem}

Finally we consider the case where $T$ is a bipartite graph and $H$ is a
tree. For a tree $H$, any $H$-free graph can have at most a linear
number of edges. Therefore, by a theorem proved in \cite{Al}, the
maximum possible number of copies of any bipartite graph $T$ in it is
bounded by $O(n^{\alpha(T)})$, where $\alpha(T)$ is the size of a
maximum independent set in $T$. Using the next definition we
characterize the cases in which $ex(n,T,H)=\Theta(n^{\alpha(T)})$.

\begin{definition} An edge cover of a graph $T$ 
(with no isolated vertices) is a set $\Gamma\subset
E(T)$ such that for each vertex $v\in V(T)$ there is an edge $e\in
\Gamma$ for which $v\in e$. Call an edge-cover minimum if it has the
smallest possible number of edges.

A set of vertices $U \subset V(T)$ is called  a $U(\Gamma)$-set if each
connected component of $T\setminus U$ intersects exactly one edge of
$\Gamma$, and the number of these connected components is $|\Gamma |$.
\end{definition}

Note that if $\Gamma$ is an edge cover of $T$ and $U$ is a
$U(\Gamma)$ set, then any connected component of $T\setminus U$
is either an edge of $\Gamma$ or a single vertex.

\begin{theorem} 
\label{thm:TreeAndBi} 
Let $T$ be a bipartite graph with no isolated vertices and
let $H$ be a tree on $h$ vertices. 
Then the following are equivalent: \begin{enumerate}
\item $ex(n,T,H)=\Theta(n^{\alpha(T)})$ \item For any minimum edge-cover
$\Gamma$ of $T$ there is a choice of a $U(\Gamma)$-set $U$ such that the
$(U,h)$ blow-up of $T$ does not contain a copy of $H$, \item For some
minimum edge cover $\Gamma$ of $T$ there is a choice of a
$U(\Gamma)$-set $U$ such that the $(U,h)$ blow-up of $T$ does not
contain a copy of $H$. \end{enumerate} 
\end{theorem}

It is worth noting that for $T\ne K_2$ the function $ex(n,T,H)$ behaves
very differently from its well studied relative $ex(n,H)=ex(n,K_2,H)$.
In particular, it is easy to see that for any graph $H$ with at least
$2$ edges, if $2H$ denotes the vertex disjoint union of two copies of
$H$, then $ex(n,H)$ and $ex(n,2H)$ have the same order of magnitude. In
contrast, if, for example, $H=C_5$ then by (\ref{e11}),
$ex(n,K_3,H)=\Theta(n^{3/2})$ and it is not difficult to show that
$ex(n,K_3,2H)=\Theta(n^2)$.  Similarly, it is known that for any graph
$H$, $ex(n,H)$ is either quadratic in $n$ or is at most
$n^{2-\epsilon(H)}$ for some fixed $\epsilon(H) >0$, whereas it is not
difficult to deduce from the results of Ruzsa and Szemer\'edi in
\cite{RSz} that for the graph $H$ consisting of two triangles sharing an
edge $n^{2-o(1)} \leq ex(n,K_3,H) \leq o(n^2)$ as shown in Section
\ref{sec:triangles}.

The rest of this paper is organized as follows.  In Section 2 we
consider the dense case, describing the simple characterization of pairs
of graphs $T$ and $H$ so that $ex(n,T,H)=\Theta(n^t)$ with $t$ being the
number of vertices of $T$, and establishing an Erd\H{o}s-Stone type
result for $ex(n,K_t,H)$ when $\chi(H)>t$. In Section 3 we study the
case $T=K_3$, proving Theorem \ref{t11} and Proposition \ref{p13}. The
proof of Theorem \ref{thm:K_mK_s,t} is given in Section 4, together with
several related results, and the proofs of Theorems \ref{thm:2trees} and
\ref{thm:TreeAndBi} are described in Section 5. The final Section 6
contains some concluding remarks and open  problems.

\section{The dense case} \label{sec2}

The case where both $T$ and $H$ are complete graphs is studied by
Erd\H{o}s in \cite{Er62} where he determines that: $$
ex(n,K_t,K_k)=\sum_{0\leq i_1\leq \dots \leq i_t\leq k-2}
\prod_{r=1}^{t} \lfloor \frac{n+i_r}{k-1} \rfloor. $$ A similar (though
less accurate) result can be obtained for general graphs. We proceed
with the simple details.

An $s$ blow-up of a graph  $H$ is the graph obtained by replacing each
vertex $v$ of $H$ by an independent set $W_v$ of size $s$, and each edge
$uv$ of $H$ by a complete bipartite graph between the corresponding two
independent sets $W_u$  and $W_v$.

As this is going to be useful throughout the paper, denote the number of
copies of a fixed graph $H$ in a graph $G$ by $\mathcal{N}(G,H)$.
\begin{proposition} \label{prop:ChromNum} Let $T$ be a fixed graph with
$t$ vertices. Then $ex(n,T,H) =\Omega(n^t)$ iff $H$ is not a subgraph of
a blow-up of $T$. Otherwise, $ex(n, T,H) \leq n^{t-\epsilon(T,H)}$ for
some $\epsilon(T,H) >0$. \end{proposition}

\begin{proof} If $H$ is not a subgraph of a blow-up of $T$ then the
graph $G$ which is the $\ell=\lfloor n/t \rfloor$-blow-up of $T$
contains no copy of $H$ and yet includes at least $\ell^t = \Omega(n^t)$
copies of $T$. This establishes the first part of the proposition.

To prove the second part, assume that $H$ is a subgraph of the
$s$-blow-up of $T$. We have to show that in this case any $H$-free graph
$G=(V,E)$ on $n$ vertices  contains less than $n^{t-\epsilon}$ copies of
$T$ for some $\epsilon=\epsilon(T,H) >0$. Indeed, suppose that $G$
contains $m$ copies of $T$. Let $V=V_1 \cup V_2 \cup \cdots \cup V_t$ be
a random partition of $V$ into $t$ pairwise disjoint classes. Let
$u_1,u_2 , \ldots ,u_t$ denote the vertices of $T$. Then the expected
number of copies of $T$ in which $u_i$ belongs to $V_i$ for all $i$ is
$m/t^t$. Thus we can fix a partition $V=V_1 \cup V_2 \cup \cdots \cup
V_t$ so that the number of such copies of $T$ is at least $m/t^t$.
Construct a $t$-uniform, $t$-partite hypergraph  on the classes of
vertices $V_1, \ldots ,V_t$ by letting a set of vertices $v_1,\ldots
,v_t$ with $v_i \in V_i$ be an edge iff $G$ contains a copy of $T$ on
these vertices, where $v_i$ plays the role of $u_i$ for each $i$.
Therefore, this hypergraph contains at least $m/t^t$ edges. By a well
known result of Erd\H{o}s \cite{Er}, if the number of edges exceeds
$n^{t-\epsilon}$ for an appropriate $\epsilon= \epsilon(t,s) >0$, then
this hypergraph contains a complete $t$-partite hypergraph with classes
of vertices $U_i \subset V_i$, $|U_i|=s$ for all $i$. This gives an
$s$-blow-up of $T$ in the original graph $G$, providing a copy of $H$ in
it, contradiction. It follows that $m \leq t^t n^{t-\epsilon}$,
completing the proof. \end{proof} When $T=K_t$ is a complete graph, $H$
is not a subgraph of any blow-up of $T=K_t$ if and only if $\chi(H)>t$.
In this case it is not difficult to determine the asymptotic value of
$ex(n,T,H)$ up to a lower order additive term, as we show next. Note
that the case $t=2$ is the classical result of Erd\H{o}s and Stone
\cite{ES}.

\begin{proposition} For any graph $H$, $ex(n,K_t,H)=\Omega(n^t)$ if and
only if $\chi(H)>t$. Furthermore, if indeed $\chi(H)=k>t$ then
$ex(n,K_t,H)=(1+o(1))\binom{k-1}{t}(\frac{n}{k-1})^t$ \end{proposition}

\begin{proof} The first part follows directly from Proposition
\ref{prop:ChromNum}. To prove the second part fix $t$ and $H$, and
suppose that $\chi(H)=k>t$. We have to show that
$ex(n,K_t,H)=(1+o(1))\binom{k-1}{t}(\frac{n}{k-1})^t$.

The lower bound is obtained by taking a T\'uran graph with no copy of
$K_{k}$. To prove the upper bound, assume $G$ is an $H$-free graph on
$n$ vertices satisfying $\N(G,K_t)=ex(n,K_t,H)$. By the previous
proposition \ref{prop:ChromNum}, as $G$ is $H$-free $\N(G,K_k)\leq
ex(n,K_k,H)=o(n^k)$.

Using the graph removal lemma (as stated in \cite{AFKS} following
\cite{RSz} and improved in \cite{Fo}) we can remove $o(n^2)$ edges from
$G$ and get a new graph $G'$ which is $K_k$-free. The removal of
$o(n^2)$ edges from $G$ can remove at most $o(n^2)O(n^{t-2})=o(n^t)$
copies of $K_t$, thus $\N(G',K_t)=(1+o(1))\N(G,K_t)$. As $G'$ is $K_k$
free one has $\N(G',K_t)\leq ex(n,K_t,K_k)=\sum_{0\leq i_1\leq \dots
\leq i_t\leq k-2} \prod_{r=1}^{t} \lfloor \frac{n+i_r}{k-1} \rfloor$.
This yields the needed result. \end{proof}

\section{Maximizing the number of triangles} \label{sec:triangles}

\subsection{ Extended friendship graphs   }

In this subsection we prove Theorem \ref{t11}. Here and throughout the
paper, we often do not make any serious attempt to optimize the absolute
constants. We also assume, whenever this is needed, that $n$ is
sufficiently large.

We first prove two lemmas.

\begin{lemma} \label{l21} Let $G=(V,E)$ be a graph with at least
$(9c-15)(c+1)n$ triangles and at most $n$  vertices, then it contains a
copy of $F_c$. \end{lemma}

\begin{proof}

Take a maximal set of edge-disjoint triangles in $G$, if they contain a
subset of size at least $c$ touching the same vertex then we are done.
Otherwise, one can color these triangles with $3(c-2)+1=3c-5$ colors so
that no two triangles with the same color share a vertex (by simply
coloring each triangle with the smallest available color). Each triangle
in our original graph $G$ shares an edge with one of these colored
triangles, as they form a maximal set, so there is a set of unicolored
triangles with at least $\frac{(9c-15)(c+1)n}{3c-5}=3(c+1)n$ triangles
sharing edges with one of them (where here we are counting the colored
triangles too).

Focusing on the triangles colored in this color and the ones sharing
edges with them, note that there are at least $3(c+1)n$ of those
organized in clusters, with each cluster consisting of one (colored)
central triangle and all others sharing an edge with it. There are at
most $n/3$ central triangles and hence more than $3cn$ triangles are not
central, thus having two vertices in the center and one outside. Call
the outside vertex the external one. There are $3cn$ of them, so there
must be a vertex $v\in V$ which is an external vertex of at least $3c$
triangles. At most $3$ triangles from each cluster can share an external
vertex, so there are $c$ triangles from different clusters sharing this
vertex, and this is the only vertex they share. These $c$ triangles form
a copy of $F_c$, as needed. \end{proof}

\begin{lemma} \label{l22} For every $k>3$ and $n$ large enough there is
a graph $G$ on $n$ vertices with at least $\Omega(n^{1+\frac{1}{k-1}})$
triangles and no cycles of length $i$ for any $i$ between $4$ and $k$.
\end{lemma}

We note that the exponent here can be slightly improved, at least for
some values of $k$. In particular, for $k=4$ the best possible value is
$(1/6+o(1))n^{3/2}$, as can be shown using the Erd\H{o}s-R\'enyi graph
\cite{ER}, or Theorem \ref{Thm:K3vsK_2,t} below with $t=2$. For our
purposes here, however, the above estimate suffices.

\begin{proof} Let $G'$ be a random graph on a  fixed set of $n$ labeled
vertices obtained by choosing, randomly and independently, each of the
${n \choose 3}$ potential triangles on the set of vertices to form a
triangle in $G'$ with probability $p=\frac{1}{2}n^{-\frac{2k-3}{k-1}}$.
Let $X$ be the random variable counting the number of triangles picked,
and for $2\leq i \leq k$ let $Y_i$ denote the random variable counting
the number of cycles of length $i$ in which each edge comes from a
different triangle. (In particular, $Y_2$ counts the number of pairs of
selected triangles that share two vertices).

Note that if we remove one of our chosen triangles from each such cycle,
then the resulting graph will contain no cycle of length between $4$ and
$k$. Indeed, if we have such a cycle using two edges of one triangle
then replacing those by the third edge  will create a shorter cycle,
that cannot exist by assumption. Similarly, a cycle of length $4$ cannot
be created by two triangles if we leave no pair of triangles sharing two
vertices. Put $Z=X-\sum_{i=2}^{k}Y_i$, and note that  it is enough to
show that the expectation of $Z$ is at least $\Omega(n^{1+1/(k-1)})$.
Indeed, if this is the case, then there is a graph $G'$ for which the
value of $Z$ is at least $\Omega(n^{1+1/(k-1)})$. Fixing such a graph
and omitting a triangle from each of the short cycles counted by the
variables $Y_i$ generates a graph $G$ with the desired properties. Since
$\E(X)={n \choose 3} p$ and $$ \E(Y_i)=\frac{n\cdot (n-1)\dots
(n-i+1)(n-2)^i}{2i}p^{i} \leq \frac{(n^2p)^i}{2i}=\frac{n^{i/(k-1)}}{i
2^{i+1}} $$ a simple computation shows that $\E(Z) \geq
(1+o(1))(1/12-1/128)n^{1+1/(k-1)}$, as needed. \end{proof}

We can now prove Theorem \ref{t11}.

\begin{proof}[Proof of Theorem \ref{t11} ] We start by showing that
$ex(n,K_3,H)$ is linear in $n$ for any extended friendship graph. Let
$H$ be an extended friendship graph with $h$ vertices and let $G$ be a
graph on $n$ vertices with at least $c(H)n$ triangles, where $c(H)=
10h^2$. We show that $G$ must contain a copy of $H$.

We first show that $G$ contains a  subgraph with minimum degree at least
$h$. As long as there is a vertex in $G$ of degree smaller than $h$,
omit it. This process must terminate with a nonempty graph containing
more than $9h^2n$ triangles, since the total number of triangles  that
can be omitted this way is smaller than  ${h  \choose 2}n<h^2 n$. We can
thus assume that the minimum degree in $G$ is at least $h$, and that it
has at most $n$ vertices and at least $9h^2 n $ triangles.

By Lemma \ref{l21} $G$ contains a copy of the 2-core of $H$. This copy
can be extended to a copy of $H$. Indeed, if $H$ is disconnected add to
it edges to make it connected (keeping the $2$-core intact). We can now
embed the missing vertices of $H$ in $G$ one by one, starting with the
2-core and always adding a vertex with exactly one neighbor in the
previously embedded vertices. Since the minimum degree in $G$ is at
least $h$  this can be done, providing the required copy of $H$.

To complete the proof of the theorem we have to show that if $H$ is not
a subgraph of an extended friendship graph then there is a graph $G$
with $n$ vertices and $\omega(n)$ triangles containing no copy of $H$.
Note that $H$ is not a subgraph of an extended friendship graph iff it
either contains a cycle of length greater than $3$ or it contains two
vertex disjoint triangles. In the first case, Lemma \ref{l22} provides a
graph $G$ with a superlinear number of triangles containing no copy of $H$.

For the second case let $G$ be the complete $3$-partite graph $K_{1,
\lfloor \frac{n-1}{2} \rfloor, \lceil \frac{n-1}{2} \rceil }$. Here all
the triangles share a common vertex, hence no two are disjoint. As the
number of triangles is $\lfloor \frac{(n-1)^2}{4} \rfloor$, this
completes the proof. \end{proof}

\begin{remark} For any connected graph $H$ with $h$ vertices, an $n$
vertex graph consisting of a disjoint union of $\lfloor n/(h-1) \rfloor$
cliques, each of size $h-1$, contains no copy of $H$ and at least
$\Omega(h^2 n)$ triangles, showing that the estimate in the proof of the
last theorem is tight, up to a constant factor. \end{remark}

\subsection{Cycles}

In this subsection we prove Proposition \ref{p13}, which (slightly)
improves the estimates in \cite{BG} and \cite{GL}. We start with the
proof of part (i).  Let $G=(V,E)$ be a $C_5$-free graph on $n$ vertices
with the maximum possible number of triangles. Clearly we may assume
that each edge of $G$ lies in at least one triangle. Put $|E|=m$ and
$\N(G,K_3)=t$. For each vertex $v\in V$ the graph spanned by its
neighborhood $N(v)$ does not contain a path of length $3$, and thus, by
a  known result of Erd\H{o}s and Gallai \cite{EG}, the number of edges
it spans satisfies $|E(N(v))|\leq d_v$, where $d_v=|N(v)|$ is the degree
of $v$. The number of edges in $N(v)$ is exactly the number of triangles
containing $v$ and therefore \begin{equation} \label{e41} t\leq
\frac{\sum_v d_v}{3}=\frac{2m}{3} \end{equation}

Color the vertices of $G$ randomly and independently, where each vertex
is blue with probability $p$ (which will be chosen later to be $p=1/3$)
and red with probability $1-p$. For each edge $e=uv$ of $G$ choose
arbitrarily one vertex $w=w(e)$ such that $u,v,w$ form a triangle.
Denote by $E'$ the set of edges $e=uv$ of $G$  so that both $u$ and $v$
are colored blue and $w$ is colored red, and denote by $V'$ the set of
all blue vertices. Note that the graph $(V',E')$ on the blue vertices
contains no $C_4$ since otherwise each edge of this $C_4$ forms a
triangle together with a red vertex, providing a copy of $C_5$ in $G$,
which is impossible. Therefore $$ |E'|\leq ex(|V'|,C_4) =
(\frac{1}{2}+o(1))|V'|^\frac{3}{2}. $$ Taking expectation in both sides
and using linearity of expectation and the fact that the binomial random
variable $|V'|$ is tightly concentrated around its mean we get $$
p^2(1-p)m\leq\E(|E'|)\leq (\frac{1}{2}+o(1))(np)^\frac{3}{2}. $$ This is
because for each edge  $uv$, the probability it belongs to $E'$ is
$p^2(1-p)$. Thus $$ m\leq
(\frac{1}{2}+o(1))n^{\frac{3}{2}}\frac{1}{\sqrt{p}(1-p)}. $$ Since the
right hand side is minimized when $p=\frac{1}{3}$ select this $p$ to
conclude that $$ m\leq
(\frac{1}{2}+o(1))n^{\frac{3}{2}}\frac{3\sqrt{3}}{2}. $$ Plugging into
(\ref{e41}) we get $$ t\leq (\frac{1}{2}+o(1))n^{\frac{3}{2}}\sqrt{3}
=\frac{\sqrt{3}}{2}n^{\frac{3}{2}}+o(n^{\frac{3}{2}}) $$ as needed.
\hfill  $\Box$

The proof of part (ii) of Proposition \ref{p13} is similar. Here we do
not optimize the value of the probability $p$ and simply take $p=1/2$,
for small values of $k$ the result can be slightly improved. To get the
precise statement  as stated in the proposition we use an additional
trick. The details follow. 

Let $G=(V,E)$ be a $C_{2k+1}$-free graph on $n$ vertices with the
maximum possible number of triangles. As before, assume that each edge
of $G$ lies in at least one triangle,  and for each edge $e=uv$ of $G$
choose  a vertex $w=w(e)$ so that $u,v,w$ form a triangle in $G$. Put
$|E|=m$ and $\N(G,K_3)=t$. Since the neighborhood of any vertex $v$ of
$G$ contains no path on $2k$ vertices, the Erd\H{o}s-Gallai theorem
implies that it contains at most $(k-1)d_v$ edges, implying that
\begin{equation} \label{e42} t\leq \frac{\sum_v
(k-1)d_v}{3}=\frac{2(k-1)m}{3} \end{equation}

Split the vertices of $G$ into $m=\lceil n/2 \rceil$ disjoint subsets,
where if $n$ is even each subset is of size $2$ and otherwise one subset
is of size $1$. If a subset chosen is an edge $uv$ of the graph $G$, we
ensure that if $w=w(uv)$ then $u=w(vw)$ and $v=w(uw)$. As the subsets
are disjoint, it is easy to check that such a choice is possible. Now
color the vertices randomly red and blue: in each subset one vertex is
colored red and the other is blue (where each of the two possibilities
are equally likely). If $n$ is odd then  the vertex in the  last class
gets a random color.  As before, let $E'$ denote the set of edges $e=uv$
of $G$  so that both $u$ and $v$ are colored blue and $w=w(e)$ is
colored red, and denote by $V'$ the set of all blue vertices. The graph
$(V',E')$ contains no $C_{2k}$ since otherwise  we get a copy of
$C_{2k+1}$ in $G$, which is impossible. Thus \begin{equation}
\label{e43} |E'|\leq ex(|V'|,C_{2k}) \leq ex(\lceil n/2 \rceil,C_{2k})
\end{equation} since here $|V'|$ is always of cardinality either $\lceil
n/2 \rceil$ or $\lfloor n/2 \rfloor$.

We claim that the expected  cardinality of  $E'$ is at least $m/8$.
Indeed, if for an edge $uv$ with $w=w(uv)$ no pair of the three vertices
$u,v,w$ lie in a single subset, then the probability that $u,v$ are blue
and $w$ is red is exactly $1/8$. For the other edges note that if $uv$
forms one of our subsets and $w=w(uv)$, then the probability that $uv$
lies in $E'$ is $0$, but the probability that $uw$ lies in $E'$ is $1/4$
and so is the probability that $vw$ lies in $E'$. Hence the contribution
from these three edges to the expectation of $|E'|$ is $2/4>3/8$.
Linearity of expectation thus implies  that the expected value of $|E'|$
is at least $m/8$ and thus by (\ref{e43}), $m/8 \leq ex(\lceil n/2
\rceil,C_{2k})$,  and by (\ref{e42}) $$ t=\N(G,K_3) \leq
\frac{16(k-1)}{3} ex(\lceil n/2 \rceil,C_{2k}) $$ completing  the proof.
\hfill $\Box$


\begin{remark} Bondy and Simonovits \cite{BS} proved that
$ex(n,C_{2k})\leq O(kn^{1+\frac{1}{k}})$. This has recently been
improved by Bukh and Jiang \cite{BJ} to $ex(n,C_{2k}) \leq O( \sqrt {k
\log k}~ n^{1+\frac{1}{k}})$. Thus the upper bound obtained from the
above proof is $ex(n,K_3,C_{2k+1}) \leq O(k^{3/2} \sqrt {\log k}~
n^{1+1/k})$. \end{remark}

\subsection{Books}

An $s$-book is the graph  consisting of $s$ triangles, all  sharing one
edge.

\begin{prop} \label{p44} For each fixed $s \geq 2$, if $H=H(s)$ is the
$s$-book then $n^{2-o(1)} \leq ex(n,K_3,H)=o(n^2)$ \end{prop}

\begin{proof} The lower bound follows from the construction of Ruzsa and
Szemer\'edi \cite{RSz}, based on Behrend's construction \cite{Be} of
dense subsets of the first $n$ integers that contain no three term
arithmetic progressions. This construction gives graphs on $n$ vertices
with $$ m=\frac{n^2}{e^{O(\sqrt {\log n})}}=n^{2-o(1)} $$ edges in which
every edge is contained in a unique triangle. Therefore these graphs
contain no $2$-book, and hence no $s$-book, showing that $$
ex(n,K_3,H(s)) \geq m/3 \geq \frac{n^2}{e^{O(\sqrt {\log n})}}
=n^{2-o(1)}. $$

The upper bound follows from the triangle removal lemma proved in
\cite{RSz}. If $G$ is a graph on $n$ vertices containing $t$ triangles
and no copy of $H(s)$, then every edge is contained in at most $s-1$
triangles.  Therefore, one has to remove at least $t/(s-1)$ edges of $G$
in order to destroy all triangles. It follows that if $t \geq \epsilon
n^2$ then, by the triangle removal lemma, the number of triangles in $G$
is at least $\delta n^3$ for some $\delta=\delta(\epsilon,s)>0$ , and
thus, by averaging, $G$ contains an $r$-book for $r \geq 2 \delta n >s$,
contradiction. Thus $t=o(n^2)$, as needed. \end{proof}

\section{Complete graphs and complete bipartite graphs}
\label{Sec:CompVSBi}

In this section we consider the cases in which $T$ and $H$ are either
complete or complete bipartite graphs. Note that when both $T$ and $H$
are complete graphs the precise value of $ex(n,T,H)$ is known, as
mentioned in Section \ref{sec2}. The following argument suffices to
provide the precise value of $ex(n,K_{a,b},K_t)$. If $u$ and $v$ are two
non-adjacent vertices in a $K_t$-free graph $G$, then by making the set
of neighbors of $u$  identical to that of $v$ (or vice versa), the graph
stays $K_t$-free, and one can always choose one of these modifications
to get a graph containing at least as many copies of $K_{a,b}$ as $G$.
This is because every copy of $K_{a,b}$ in $G$ that contains both $u$
and $v$ remains a copy in the modified graph as well.  Repeating this
procedure until every two nonadjacent vertices have the same
neighborhoods we get a complete multipartite graph with $n$ vertices and
at most $t-1$ color classes, and one can now optimize the sizes of the
color classes to obtain the maximum possible number of copies of
$K_{a,b}$.   (Note that this optimum is not necessarily obtained for
equal or nearly equal color classes. Note also that the precise argument
here requires to ensure the process above converges. To do so we can
assign all potential copies of $K_{a,b}$ nearly equal algebraically
independent weights, and always select the modification that maximizes
the total weight of all copies obtained. It is not difficult to argue
that for large $n$, in the extremal graph for any two distinct vertices,
there are copies of $K_{a,b}$ containing exactly one of them,  and
therefore in the above process the total weight keeps increasing and it
must converge.) The same argument shows that for any complete
multipartite graph $T$ with less than $t$ color classes, the extremal
graph giving the value of $ex(n,T,K_t)$ is itself a complete
multipartite graph.

As to the case when $H=K_{s,t}$ for $s\leq t$ and $T=K_m$, note that if
$H=K_{1,t}$, then it is a star and avoiding it in a graph means bounding
the degrees of the vertices. Thus to find $ex(n,K_m,K_{1,t})$ for $m\leq
t$ first note that as each vertex has degree at most $t-1$ the number of
$K_m$s is at most $\frac{n}{m}\binom{t-1}{m-1}$. On the other hand, if
$n$ is divisible by $t$ then taking $\frac{n}{t}$ disjoint copies of
$K_t$ will yield $\frac{n}{t}\binom{t}{m}=\frac{n}{m}\binom{t-1}{m-1}$
copies of $K_m$. If $n$ is not divisible by $t$ a similar bound can be
achieved by taking $\lfloor \frac{n}{t} \rfloor$ copies of $K_t$ and a
clique on the remaining vertices. In  \cite{GLS} it is conjectured that
the above construction is optimal, and this is proved for some specific
cases.

For general $m,s,t$ we  start by proving Theorem \ref{thm:K_mK_s,t}.
After that we prove another bound for cases that do not satisfy the
assumptions of the theorem and then  establish tighter results for
several values of $s,t$ when $T=K_3$.

To prove Theorem \ref{thm:K_mK_s,t} in a more precise form we prove two
lemmas, one for the upper bound and one for the lower. \begin{lemma}
\label{lem:UpperBoundK_m,K_s,t} For any fixed $m\geq 2$ and $t\geq s
\geq m-1$ \[ ex(n,K_m,K_{s,t})\leq
(\frac{1}{m!}+o(1))(t-1)^\frac{m(m-1)}{2s}n^{m-\frac{m(m-1)}{2s}} \]
\end{lemma}

\begin{proof} We  apply induction on $m$.

For $m=2$, by the theorem of K\"ovari, S\'os and Tur\'an in \cite{KST}:
$$ ex(n,K_2,K_{s,t})=ex(n,K_{s,t})\leq
(\frac{1}{2}+o(1))(t-1)^{\frac{1}{s}}n^{2-\frac{1}{s}}. $$ This serves
as our base case.

Now assume we have proved this for $m$ and let us prove it for $m+1$. In
what follows it will be convenient to use the means-inequality: For each
$r<s$ and positive reals $x_1, \ldots ,x_n$: $$ \sum_{i=1}^n x_i^r \leq
n^{1-r/s} (\sum_{i=1}^n x_i^s)^{r/s}. $$

Let $G=(V,E)$ be a $K_{s,t}$ free graph on $n$ vertices, and let us
bound the number of copies of $K_{m+1}$ in it. For each $v\in V$ we know
that its neighborhood $N(v)$ does not contain any copy of $K_{s-1,t}$.
By the induction assumption we can bound the number of copies of $K_{m}$
in $N(v)$: \[ \mathcal{N}(N(v),K_m)\leq ex(d_v,K_m,K_{s-1,t})\leq
(\frac{1}{m!}+o(1))(t-1)^\frac{m(m-1)}{2(s-1)}
d_{v}^{m-\frac{m(m-1)}{2(s-1)}} \]

By bounding the number of $K_m$ in each $N(v)$ we can bound the number
of $K_{m+1}$ in $G$ resulting in:

\begin{align} \mathcal{N}(G,K_{m+1}) \leq & \frac{1}{m+1}
(\frac{1}{m!}+o(1))(t-1)^{\frac{m(m-1)}{2(s-1)}} \sum_v
d_{v}^{m-\frac{m(m-1)}{2(s-1)}} \nonumber \\ \leq &
(\frac{1}{(m+1)!}+o(1)) (t-1)^{\frac{m(m-1)}{2(s-1)}} (\sum_v
d_{v}^{s})^{\frac{m(2s-m-1)}{2s(s-1)}}n^{1-\frac{m(2s-m-1)}{2s(s-1)}} \\
\leq & (\frac{1}{(m+1)!}+o(1)) (t-1)^{\frac{(m+1)m}{2s}}
n^{\frac{m(2s-m-1)}{2(s-1)} + 1-\frac{m(2s-m-1)}{2s(s-1)}} \\ 
= & (\frac{1}{(m+1)!}+o(1)) (t-1)^{\frac{(m+1)m}{2s}}
n^{(m+1)-\frac{(m+1)m}{2s}} \nonumber \end{align}

Here we used the means inequality to get the first inequality (an easy
calculation shows that $m-\frac{m(m-1)}{2(s-1)} <s$). To bound the sum
$\sum_v d_v^s$ we used the fact that the number of $s$-edged stars in
$G$ cannot exceed $\binom{n}{s}(t-1)$ because otherwise $t$ of them will
share the same $s$ leaves, creating a $K_{s,t}$. \end{proof}

\begin{lemma} \label{lem:LowerBoundK_m,K_s,t} For any fixed $m$, $s \geq
2m-2$ and $t\geq (s-1)!+1$ \[ ex(n,K_m,K_{s,t})\geq (\frac{1}{m!}+o(1))
n^{m-\frac{m(m-1)}{2s}} \] \end{lemma}

\begin{proof}

We use the projective norm-graphs constructed in \cite{ARS}, where it is
shown that $H(q,s)$ has $n=(1+o(1))q^s$ vertices, is $d=(1+o(1))
q^{s-1}$-regular, and is $K_{s,(s-1)!+1}$ free. An $(n,d, \lambda)$ graph
is a $d$-regular graph on $n$ vertices in which all eigenvalues but the
first have absolute value at most $\lambda$. As shown in \cite{AP} (see
also \cite{KS}, Theorem 4.10) the following holds: Let $G_1$ be a fixed
graph with $r$ edges, $s$ vertices and maximum degree $\Delta$. Let
$G_2$ be an $(n,d,\lambda)$ graph. If $n\gg
\lambda(\frac{n}{d})^{\Delta}$ then the number of copies of $G_1$ in
$G_2$ is $(1+o(1))\frac{n^s}{|Aut(G_1)|} (\frac{d}{n})^r$.

In our case we take $G_1=K_m$ and $G_2=H(q,s)$. By the results in
\cite{Sz} or \cite{AR} we know that the second eigenvalue, in absolute
value, of $H(q,s)$ is $q^{\frac{s-1}{2}}$, thus to get the inequality
$n\gg \lambda(\frac{n}{d})^{\Delta}$ it suffices that $m<\frac{s+3}{2}$.
Plugging the choice of $G_1,G_2$ into the result mentioned above
implies:

\begin{align*} \mathcal{N}(H(q,s),K_m)= & (1+o(1)) \frac{n^m}{m!}
(\frac{1}{q})^{\binom{m}{2}} \\ = & (\frac{1}{m!}+o(1)) (q^s-q^{s-1})^m
(\frac{1}{q})^{\binom{m}{2}} \\ =& (\frac{1}{m!}+o(1))
q^{s(m-\frac{m(m-1)}{2s})} \\ =& (\frac{1}{m!}+o(1))
n^{m-\frac{m(m-1)}{2s}} \end{align*}

\end{proof}

Note that for $m=3$ the lower bound above applies only for $s \geq 4$.
The following result provides a similar bound for $s \in {2,3}$ as well.

\begin{lemma} \label{l91} For any fixed $s \geq 2$ and $t\geq (s-1)!+1$
we have $ex(n,K_3,K_{s,t}) = \Theta(n^{3-\frac{3}{s}})$ \end{lemma}

\begin{proof} In view of the previous upper bound it suffices to show
the existence of a graph $G$ with $n$ vertices containing no copy of
$K_{s,t}$ and containing at least $\Omega(n^{3-\frac{3}{s}})$ triangles.
For this we apply again the projective norm-graphs $H(q,s)$ constructed
in \cite{ARS}, which are $K_{s,t}$ free.

The graph $H=H(q,s)$ is defined in the following way:
$V(H)=GF(q^{s-1})\times GF(q)^*$ where $GF(q)^*$ is the multiplicative
group of the $q$ element field. For $A\in GF(q^{s-1})$ define the norm
$$N(A)=A \cdot A^q \dots A^{q^{s-2}}. $$ Two vertices $(A,a)$ and
$(B,b)$ are connected in $H$ if $N(A+B)=ab$. Note that
$|V(H)|=q^s-q^{s-1}$ and $H$ is $q^{s-1}-1$ regular.

We need to show that $H(q,s)$ has the right number of triangles. As
mentioned above, the eigenvalues and multiplicities of $H(q,s)$ are
given in \cite{Sz}, \cite{AR}. These are as follows: \noindent
$q^{s-1}-1$ is of multiplicity $1$, ~~$0$ is of multiplicity $q-2$,
~~$1$ and $-1$ are of multiplicity $(q^{s-1}-1)/2$ each, and
$q^{(s-1)/2}$, $-q^{(s-1)/2}$ are of multiplicity $(q^{s-1}-1)(q-2)/2$
each. Summing the cubes of the eigenvalues we conclude that the number
of closed walks of length $3$ in $H(q,s)$ is
$(q^{s-1}-1)^3=(1+o(1))q^{3s-3}$.

A closed walk of length $3$ is not a triangle iff it contains a loop.
Fixing $A\in GF(q^t)$ the vertex $(A,x)$ has a loop iff $N(A+A)=x^2$.
There are at most $2$ solution $x$ for each given $A$. Thus there are no
more than $2q^{s-1}$ loops. A closed walk of length $3$ containing a
loop must also contain an additional edge taken twice (this additional
edge may also be the loop itself). As the graph is $q^{s-1}-1$ regular
we get at most $6q^{s-1} q^{s-1}=o(q^{3s-3})$ such walks containing a
loop. As the number of closed walks of length $3$ is $(1+o(1)) q^{3s-3}$
this is negligible and the number of triangles is
$(\frac{1}{6}+o(1))q^{3s-3}= \Theta(|V(H)|^{3-3/s})$, as needed.
\end{proof}

\begin{remark} For the special case of $s=t=3$ it can be shown that the
construction of Brown \cite{Br} gives another example of a
$K_{3,3}$-free on $n$ vertices with essentially the same number of
triangles. \end{remark}

\begin{remark}  The number of triangles in the projective norm graphs is
also computed in a recent paper of Kostochka, Mubayi and Verstra\"ete
\cite{KMV}, motivated by an extremal problem for $3$-uniform
hypergraphs. They estimate this number directly, without using the
eigenvalues. \end{remark}

For values of $s,t$ and $m$ that do not satisfy the restrictions in the
previous results we provide slightly weaker results in the following
lemmas:

\begin{lemma} \label{l53} For any fixed $m$ and $t\geq s\geq 1$ such
that $t+s>m$ \[ ex(n,K_m,K_{s,t}) \leq
(1+o(1))\frac{(m-s)!(t-1)^\frac{s-1}{2}}{m!}\binom{t-1}{m-s}
n^\frac{s+1}{2} \] \end{lemma}

\begin{proof}

We apply induction on $s$. As the base case take $s=1$. In this case the
fact that $G$ is $K_{1,t}$ free implies that the degrees of all vertices
are at most $t-1$. Thus each vertex can take part in no more than
$\binom{t-1}{m-1}$ copies of $K_m$ and hence \[ ex(n,K_m,K_{1,t})\leq
\frac{1}{m}\binom{t-1}{m-1}n \] Note that if $t\mid n$ then this bound
is achieved by the disjoint union of $\frac{n}{t}$ pairwise vertex
disjoint copies of $K_t$.

Assuming the result for $s-1$ we prove it for $s$. If $G$ is $K_{s,t}$
free, then for each $v\in V$ its neighborhood cannot contain a copy of
$K_{s-1,t}$. By the induction hypothesis this bounds the number of
copies of $K_{m-1}$ by $$ (1+o(1))
\frac{(m-s)!(t-1)^\frac{s-2}{2}}{(m-1)!} \binom{t-1}{m-s}
d_v^\frac{s}{2} $$ where $d_v$ is the degree of $v$. This is clearly
also the number of copies of $K_m$ containing $v$. Therefore,

\begin{align} \N(G,K_m)& \leq \frac{1}{m} (1+o(1))\sum_v
\frac{(m-s)!(t-1)^\frac{s-2}{2}}{(m-1)!}\binom{t-1}{m-s}
d_v^{\frac{s}{2}} \nonumber \\ &\leq
(1+o(1))\frac{(m-s)!(t-1)^\frac{s-2}{2}}{m!}\binom{t-1}{m-s}(\sum_v
d_v^{s})^{\frac{1}{2}}n^\frac{1}{2} \\ &\leq
(1+o(1))\frac{(m-s)!(t-1)^\frac{s-1}{2}}{m!} \binom{t-1}{m-s}
n^{\frac{s+1}{2}} \end{align}

\noindent where we get (8) from the means inequality and (9) from the
fact that we cannot have more than $\binom{n}{s}(t-1)$ copies of $s$
stars in $G$. \end{proof}

Note that unlike Theorem \ref{thm:K_mK_s,t} to get the bound in Lemma
\ref{l53} we need to assume nothing but the obvious fact that $K_m$ does
not contain a copy of $K_{s,t}$. On the other hand for every $m,s \in
\mathbb{N}$ one has $\frac{s+1}{2}\geq m-\frac{m(m-1)}{2s}$ where we
have an equality when $s=m-1$ and $s=m$. Thus when $s<m-1$ we must use
Lemma \ref{l53}, but if $s\geq m-1$ Lemma \ref{lem:UpperBoundK_m,K_s,t}
gives a better upper bound.

\begin{lemma} \label{l54} For any $m$ and $t>m-2>1$
$$ex(n,K_m,K_{2,t})\geq \frac{1}{4} m^{\frac{-4m}{3}}n^{\frac{4}{3}}$$
\end{lemma}

\begin{proof}

In \cite{LV} Lazebnik and Verstra\"ete show that there exists an
$m-$uniform hypergraph $H$ on $n$ vertices, with at least $\frac{1}{4}
m^{\frac{-4m}{3}}n^{\frac{4}{3}}$ hyperedges and with girth at least
$5$. Let $G$ be the graph obtained from $H$ by replacing each hyperedge
of $H$ by a copy of $K_m$. We next observe that $G$ contains no copy of
$K_{2,t}$.

Assume towards a contradiction that $G$ contains a copy of $K_{2,t}$. As
$t>m-2$ the copy of $K_{2,t}$ cannot be contained in a single $K_m$ and
so there must be at least two edges in it that come from two different
$K_m$s. These two edges are a part of a $C_4$ hence in the
hypergraph $H$ this $C_4$ has vertices in at least two hyperedges. Thus
$H$ must contain a cycle of length between $2$ and $4$ in contradiction
to the assumption that $H$ has girth at least $5$.  Therefore $G$ is
$K_{2,t}$ free with at least $\frac{1}{4}
m^{\frac{-4m}{3}}n^{\frac{4}{3}}$ copies of $K_m$, as needed.
\end{proof}

Finally for $s=2$ we can determine the asymptotic behavior of
$ex(n,K_3,K_{2,t})$ up to a lower order term, as shown next.
\begin{theo} \label{Thm:K3vsK_2,t} For any fixed $t \geq 2$,
$ex(n,K_3,K_{2,t})= (1+o(1)) \frac{1}{6}(t-1)^{3/2} n^{3/2}$. \end{theo}

\begin{proof} The upper bound follows from the assertion of Lemma
\ref{lem:UpperBoundK_m,K_s,t} with $m=3$ and $s=2$. To prove the lower
bound we apply a construction of F\"uredi \cite{Fu}, extending the one
of Erd\H{o}s and R\'enyi \cite{ER}. The details follow. Let $\F$ be a
finite field of order $q$, where $t-1$ divides $q-1$, and let $h$ be a
nonzero element of $\F$ that generates a multiplicative subgroup
$A=\{h,h^2,...,h^{t-1}=1\}$ of order $t-1$ in $\F^*$. The vertices of
the graph $G=G(\F,t-1)$ are all nonzero pairs in $(\F \times \F)$, where
two pairs $(a,b)$ and $(a',b')$ are considered equivalent if for some $
h^{\alpha}\in A$, $h^{\alpha}a=a' \mbox{ and } h^{\alpha}b=b'$. Two
vertices $(a,b),(c,d)$ are connected if $ac+bd\in A$. The number of
vertices of $G$ is $n=(q^2-1)/(t-1)$ and it is not difficult to check
that it is regular of degree $q$, where here each loop adds one to the
degree. Indeed, for a fixed vertex $(b,c)$ and for each $h^{\alpha} \in
A$ there are exactly $q$ solutions $(x,y)$ to the equation
$bx+cy=h^{\alpha}$, and as any neighbor $(x,y)$ of $(b,c)$ is obtained
this way $t-1$ times, by our equivalence relation, the graph is
$q$-regular.  Note that there is a (unique) loop at a vertex $(x,y)$ iff
$x^2+y^2 \in A$. For each fixed $h^{\alpha} \in A$ and each fixed $x \in
\F$ there are at most $2$ solutions for $y$, showing that the number of
loops is at most $2q(t-1)/(t-1)=2q$ (it is in fact smaller, but this
estimate suffices  for us).

It thus follows that the number of edges of $G$ (without the loops) is
$m=(\frac{1}{2}+o(1))q^3/(t-1)= (\frac{1}{2}+o(1))\sqrt{t-1}~n^{3/2}$.

We claim that any two distinct vertices $(a,b)$ and $(c,d)$ of $G$ have
exactly $t-1$ common neighbors (if there is a loop in one of these
vertices and they are adjacent, this counts as a common neighbor).
Indeed, the vertex $(x,y)$ is a common neighbor if for some $0\leq
\alpha,\beta \leq t-2$ \begin{align*} ax+by&=h^{\alpha}\\
cx+dy&=h^{\beta}. \end{align*} These two  equations are linearly
independent, and hence there is a unique solution for each choice of
$\alpha,\beta$. As the number of choices for $\alpha$ and $\beta$ is
$(t-1)^2$, and every common neighbor is counted this way $t-1$ times,
the claim follows.

By the above claim, $G$ is $K_{2,t}$-free. In addition, since the
endpoints of each edge in it have $t-1$ common neighbors, each edge is
contained in $t-1$ triangles (including the degenerated ones containing
a loop). The number of triangles containing a loop is smaller than
$2q^2$ which is far smaller than the number of edges
$m=\Theta(q^3/(t-1))$. Therefore, the number of triangles is $$ (1+o(1))
\frac{1}{3} m (t-1) =(1+o(1))\frac{1}{6}(t-1)\sqrt{t-1}~n^{3/2} $$
completing the proof. \end{proof}

We conclude the section by considering the case $T=K_{a,b}$ and
$H=K_{s,t}$ where we establish the following. \begin{prop} \label{p55}
(i) If $a \leq s \leq t$ and $a \leq b<t$ then $$ ex (n,K_{a,b},
K_{s,t}) \leq (1+o(1)) \frac{1}{a! (b!)^{1-a/s} } {{t-1} \choose
b}^{a/s} n^{a+b-ab/s}, $$ and if $a=b$ the above bound can be divided by
$2$.

\noindent (ii) If $(a-1)!+1 \leq b < (s+1)/2$ then for all $t \geq s,~~$
$ex(n,K_{a,b},K_{s,t}) =\Theta(n^{a+b-ab/s}).$ \end{prop}

\begin{proof} (i) Let $G=(V,E)$ be a $K_{s,t}$-free graph on $n$
vertices. For each subset  $B$ of $b$ vertices, let $n_B$ denote the
number of common neighbors of all vertices in $B$.  The number of copies
of $K_{a,b}$ in $G$ is clearly exactly $\sum_B {{n_B} \choose a}$ for
$b<a$, where the summation here and in what follows is over all
$b$-subsets $B$ of $V$. If $a=b$ the right hand side should be divided
by $2$. We proceed with the case $a<b$ recalling that a factor of $1/2$
can be added if $a=b$. By the means inequality, the number of copies of
$K_{a,b}$ in $G$ is at most $$ \frac{1}{a!} \sum_B n_B^a \leq
\frac{1}{a!}  {n \choose b}^{1-a/s} (\sum_B n_B^s )^{a/s} $$ $$ \leq
(1+o(1)) \frac{n^{b-ab/s}}{a!(b!)^{1-a/s}} ({{t-1} \choose b} n^s)^{a/s}
= (1+o(1)) \frac{1}{a!(b!)^{1-a/s}} ({{t-1} \choose b})^{a/s}
n^{a+b-ab/s}. $$ Here we used the fact that $\sum_B n_B^s \leq (1+o(1))
{{t-1} \choose b} n^s$ since if we have more than ${{t-1} \choose b}$
subsets of cardinality $b$ in $V$ with each of them having the same
$s$-subset among their common neighbors, then we get a copy of
$K_{s,t}$, which is impossible. \vspace{0.2cm}

\noindent (ii) The  projective norm graphs give, as in the proof of
Lemma \ref{lem:LowerBoundK_m,K_s,t}, that if $(a-1)!+1 \leq b < (s+1)/2$
then $ex(n,K_{a,b},K_{s,t}) \geq \Omega(n^{a+b-ab/s}).$ This and part
(i) supply the assertion of part (ii). \end{proof}

\section{Forbidding a fixed tree} \label{sec:HisTree}

In this section we prove Theorems \ref{thm:2trees} and
\ref{thm:TreeAndBi}.

\subsection{Proof of Theorem \ref{thm:2trees}}

Let $G=(V,E)$ be a graph, and let $T$ be a tree on a set
$V(T)=\{u_1, \ldots ,u_t\}$ of $t$ vertices. 
Let $V=V_1 \cup V_2 \cup \cdots \cup V_t$ be a partition of
$V$ into $t$ pairwise disjoint sets. Call a copy of $T$ in $G$ 
{\em proper} if its vertices are $v_1, v_2, \ldots ,v_t$ where
$v_i \in V_i$ and $v_i$ plays the role of $u_i$ in this copy.
For a subset $U \subset V(T)$ and an integer  $h$, call a 
$(U,h)$-blow-up of $T$ {\em proper} if each vertex $v$ of the blow-up 
belongs to $V_i$ if and only if it plays the role of $u_i$ in the
blow up.

Recall that  a graph is $h'$-degenerate if every subgraph of it
contains a vertex of degree at most $h'$. It is easy and well known
that if a graph is not $h'$-degenerate then it contains a copy of
any tree with $h'$ edges.

The main part of the proof is the following lemma. 

\begin{lemma} 
\label{lem:indHyp} 
For every positive integers 
$t,h,h',m$ there is a $C(t,h,h',m)$
so that the following holds. Let $T$ be a tree on a set $V(T)$ of
$t$ vertices, let $G=(V,E)$ be a graph on $n$ vertices
and let $V=V_1 \cup \cdots \cup V_t$ be a partition of $V$.
If $G$ is $h'$ degenerate and it contains at least $C(t,h,h',m)n^{m-1}$
proper copies of $T$, then it must contain a proper $(U,h)$ blow-up
of $T$, such that $T\setminus U$ has $m$ connected components.
\end{lemma}

\begin{proof} We apply induction on $m+t$.

\textbf{The base case m+t=3:} In this case $t=2$, that is, $T$ is an
edge, and $m=1$. Let $G=(V,E)$ and $V=V_1 \cup V_2$ 
be a graph and  a partition of its vertex set 
as in the statement of the lemma, and suppose it contains at least
$(h-1)^2+1$ proper copies of $T$. These copies form a bipartite
graph with vertex classes $V_1,V_2$ and hence, by K\"onig's
Theorem, it 
must contain either a star or a matching of size $h$. A matching is a
proper $(\emptyset, h)$ blow-up of $T=K_2$ 
and a star is a proper $(\{v\},h)$ blow-up,
where $v$ is one of the vertices of $T$. In both cases $T\setminus U$
has $1$ connected component. This establishes the base case.

\textbf{Induction step}  Assuming the assertion holds for any $m$
and $t$ satisfying $m+t<k$, we prove it for $m+t=k$, $(k \geq
4)$. Let $T$ be a tree on $t$ vertices, and suppose that
$V(T)=\{u_1,...,u_t\}$ where $u_1$ is a
leaf and $u_2$ is its unique neighbor.
Let $G=(V,E)$ be an $h'$-degenerate graph 
with $n$ vertices,
and let $V=V_1 \cup \ldots V_t$ be a partition of its vertex
set. Suppose that $G$ contains at least
$C(t,h,h',m)n^{m-1}$ proper copies of $T$. We have to show that it
contains a 
proper $(U,h)$ blow-up of $T$ such that $T\setminus U$ has $m$
connected components.  

For each vertex $v \in V_2$ let $d_1(v)$ be the number of its
neighbors in $V_1$.
Furthermore, put $T'=T\setminus
\{u_1\}$ and let $\mathcal{N}_{u_2}(T',v)$ be the number of copies of
$T'$ in which $v$ plays the role of $u_2$ and for each $2 < i \leq
t$ the vertex playing the role of $u_i$ lies in $V_i$.
The following clearly holds, with $C=C(t,h,h',m)$:
\begin{align*} 
Cn^{m-1} \leq
& \sum_{v\in V_2} d_1(v)\cdot \mathcal{N}_{u_2}(T',v) \\ &
=\sum_{v\in
V_2,d_1(v)\geq h} d_1(v)\cdot \mathcal{N}_{u_2}(T',v)+\sum_{v\in
V_2,d_1(v)< h} d_1(v)\cdot \mathcal{N}_{u_2}(T',v) 
\end{align*} 
One of
the summands must be at least $\frac{C}{2} n^{m-1}$. We consider both
cases.

\underline{Case 1}: $\sum_{v\in V_2,d_1(v)\geq h} d_1(v)\cdot
\mathcal{N}_{u_2}(T',v) \geq \frac{C}{2} n^{m-1}$.

If $m=1$, then there is a vertex $v \in V_2$ with $d_1(v)\geq h$ and
$\mathcal{N}_{u_2}(T',v)\geq 1$. This implies the existence of a
proper $(V(T)\setminus\{u_1\},h)$ blow-up of $T$. 
If $m>1$, as $G$ has at most $h'n$ edges (since it is
$h'$-degenerate),
$\sum_{v\in V_2,d_1(v)\geq h}
d_1(v)\leq |E(G)|\leq h'\cdot n$, and thus there must be a vertex 
$v_2$ for which
$\mathcal{N}_{u_2}(T',v_2)\geq  \frac{C}{2h'} n^{m-2}$.

Consider the induced subgraph $G'$ of $G$ on the set of vertices 
$V'=\{v_2\} \cup V_3 \cup \cdots \cup V_t$, with this partition
into $(t-1)$ disjoint sets. This graph contains
all the proper copies of  
$T'$ in which
$v_2$ plays the role of $u_2$. $G'$ is also $h'$-degenerate, and
contains
at least $\frac{C'}{2h} n^{m-2}$ proper copies of $T'$.
As $|V(T')|=t-1$ (and also $m-1<m$) we can use the induction assumption on
$G'$, and find a proper $(U,h)$ blow up of $T'$ in $G'$ 
in which $T'\setminus
U$ has $m-1$ connected components and $u_2\in U$ as there is only one
vertex, $v_2 \in V_2$, and hence only $v_2$ can play the role of
$u_2$ in $G'$. The same set $U$ thus gives the
required proper $(U,h)$ blow up of $T$, as $T \setminus U$  has all of the
connected components of $T' \setminus U'$  together with a new connected
component which is $\{u_1\}$. There is a copy of this proper 
$(U,h)$ blow up of
$T$ as we can complete the $(U,h)$ blow up of $T'$ with $h$ neighbors
of $v_2$.

\underline{Case 2}: $\sum_{v\in V_2,d_1(v)< h} d_1(v)\cdot
\mathcal{N}_{u_2}(T',v) \geq \frac{C}{2} n^{m-1}$.

Let $G'$  be the induced subgraph of $G$ on
$V'_2 \cup \cdots \cup V_t$,  where $V'_2$ is the set of
all vertices of $V_2$ satisfying
$1 \leq d_1(v)<h$. As
$\sum_{v\in V_2,1\leq d_1(v)< h} d_1(v) \mathcal{N}_{u_2}(T',v) \geq
\frac{C}{2} n^{m-1}$ and $|V(T')|=t-1$ we can use the induction
assumption and find a proper $(U',h^2)$ 
blow-up of $T'$ with $T' \setminus U'$
having $m$ connected components. It is left to complete this blow-up to
the required proper $(U,h)$ blow up of $T$.

Consider the vertices that play the role of $u_2$ in the $(U',h^2)$ blow
up of $T'$.  There are two options: either there is only one such vertex
or there are $h^2$ of them. If it is a single vertex then we complete
the blow-up into a proper 
$(U,h)$ blow up of $T$ by taking $U=U'\cup\{u_1\}$
(this is actually a $(U,h^2)$ blow-up). If there are $h^2$ of them,
consider
the bipartite graph consisting of $h^2$ edges, $1$ from each copy of
$u_2$ to an arbitrarily chosen neighbor of it in $V_1$. This graph must
contain either a matching of size $h$ or a star with $h$ edges. A
matching will leave us with the same $U=U'$ and for a star we take
$U=U'\cup \{u_1\}$. In both cases the proper 
$(U,h)$ blow up is contained in
$G$ and $T\setminus U$ has $m$ connected components. 
\end{proof}

We are now ready to prove Theorem \ref{thm:2trees}.

\begin{proof}[Proof of Theorem \ref{thm:2trees}] 
Let $H$ and $T$ be
trees with $h$ and $t$ vertices, respectively, and suppose 
$m=m(T,H)$. By the definition of $m(T,H)$ there exists a $U$
such that $T \setminus U$ has $m$ connected components, and the $(U,h)$
blow-up of $T$ 
has no copy of $H$. Using the same set $U$ define $G_T$ to be a
$(U,\frac{n-|U|}{t-|U|})$ blow-up of $T$.

We next show that $G_T$ is $H$ free and has at least 
$c_1(t,h) n^m$ copies of
$T$. If there is a copy of $H$ in $G_T$, then it uses $h$ vertices
and is thus contained in a $(U,h)$ blow up of $T$. But by the
definition of $m$ this blow-up is $H$ free, so $G_T$ must be $H$ free
too. To find $c_1(t,h)n^m$ copies of $T$, recall that $T \setminus U$
has $m$ connected components, and note that 
any choice of  a copy of each connected
component can be completed into a copy of $T$.

On the other hand, if $ex(n,T,H)>t^tC(t,h,h,m)n^m$ take a random
partition of the vertex set $V$ of the extremal graph $G$ into $t$
pairwise disjoint sets $V_1, \ldots, V_t$. Each fixed copy of $T$
becomes a proper copy with respect to this partition with
probability $1/t^t$. Thus, by linearity of expectation, the
expected number of proper copies is at least $C(t,h,h,m)n^m$, and
hence there is a partition with at least that many proper copies. 
As $G$ is $H$-free it is $h$-degenerate.
By the lemma
above it contains a $(U,h)$ blow-up of $T$ where
$T\setminus U$ has $m+1$ connected components. This contradicts the
maximality of $m(T,H)$, so $ex(n,T,H)\leq t^tC(t,h,h,m)n^m$.  
Finally 
define $c_2(t,h)=\max _{1\leq m \leq t/2} C(t,h,h,m)$. 
\end{proof}

\subsection{Proof of Theorem \ref{thm:TreeAndBi}}

We need the following well known result, see, for example, Theorem
7.3 in \cite{BM} for a proof.

\begin{lemma}
\label{rem:AlphaGamma}

In any bipartite graph with no isolated
vertices, the number of vertices in a maximum independent set is equal to
the number of edges in a minimum edge-cover. 

\end{lemma}

%

\begin{proof}[Proof of Theorem \ref{thm:TreeAndBi} ]

We prove the equivalence of the statements by deriving (2) from (1), (3)
from (2) and (1) from (3).

$1 \Rightarrow 2$ Assume $ex(n,T,H)=\Theta(n^{\alpha(T)})$. Fix any
minimum edge cover $\Gamma$ of $T$ and enumerate the vertices of $T$
arbitrarily $\{u_1,..,u_t\}$, where $t$ is the number of vertices
of $T$ . Let $G=G_T$ be an extremal graph,
that is, a graph on $n$ vertices with $cn^{\alpha(T)}$ copies of $T$ and no
copy of $H$. Enumerate the vertices of $G_T$ randomly, and  call a copy
of $T$ monotone if it is spanned by a set of vertices enumerated $i_1<
..< i_t$ where $i_j$ plays the role of $u_j$ in $T$. 
By linearity of expectation the expected number of monotone copies
of $T$ in $G$ is
$\frac{c}{t!} n^{\alpha(T)}=
c'n^{\alpha(T)}$. Fix a numbering with at least that many monotone
copies. We next show
that this graph must contain a $(U,h)$ blow up of $T$ for some
choice of a $U(\Gamma)$-set $U$.

Denote the set of 
edges in the edge cover $\Gamma$ by $\{e_1,..,e_{\alpha(T)}\}$. We
can map the monotone 
copies of $T$ to choices of edges that play the role of $\Gamma$
so there must be at least 
$c'n^{\alpha(T)}=c'n^{|\Gamma|}$ such choices (the
equality is by Lemma \ref{rem:AlphaGamma}). Consider the following
hypergraph $\HY=(V_{\HY},E_{\HY})$. The vertices $V_{\HY}$ are the edges
of $G$ and a set 
$\{v_1,...,v_{\alpha(T)}\}$ forms an edge in $E_{\HY}$ 
if the corresponding edges in $G$
span a monotone copy of $T$, 
where each edge plays its enumerated role in $\Gamma$. By the
assumption on $G$, we have $|E_{\HY}|=c'n^{\alpha(T)}$, and
$|V_{\HY}|\leq hn$. By the main theorem in \cite{Er} there is a
$K^{\alpha(T)}_{s,..,s}$ in our hypergraph, where $s=h^2$, provided
$n$ is sufficiently large. Therefore
there are disjoint sets of vertices $U_1,..,U_{\alpha(T)}\subset
V_{\HY}$ such that for any choice of $u_i\in U_i$, 
$\{u_1,...,u_{\alpha(T)}\}\in E_{\HY}$ and $|U_i|=h^2 $ for all $i$.

Returning to $G_T$ note that
the $K^{\alpha(T)}_{s,..,s}$ in our hypergraph provides
pairwise disjoint sets of edges $E_1,..,E_{\alpha(T)}\subset
E(G_T)$, each of size $s=h^2$, such that
any choice of a single edge from each set $E_i$ spans a monotone 
copy of $T$ in $G$,
where the edge from $E_i$ plays the role of the edge $e_i$ in the copy.
We next show that $E_1,..,E_{\alpha(T)}$  and the edges connecting
them in $G$ contain a $(U,h)$ blow up of
$T$, with $U$ being a $U(\Gamma)$-set.

To this end we define the set $U\subset V(T)$. First note that a
minimum 
edge cover does not contain a path of length $3$ and hence
$\Gamma$ must be a union of stars
and single edges. Define the complement $U^c$ of $U$
in the following way. If the edges
$e_{i_1},..,e_{i_k}\in E(T)$ form a star in $\Gamma$ take the
leaves 
of the star into $U^{c}$. For a single edge in $\Gamma$, say
$e_j$, consider the corresponding set $E_j$. It must contain
either a star with at least $h$ edges or a matching with
at least $h$ edges. If it contains an $h$-star, take into
$U^{c}$ the endpoint which is not the center of the star. If it contains
a matching of size $h$, take into $U^{c}$ both the vertices of
$e_j$. Finally, put
$U=V(T)\setminus U^{c}$.

It is left to show that $U$ is indeed a $U(\Gamma)$-set with the
required properties. We first show that
the connected components in $T\setminus U$ have the needed properties.
As $V(T)\setminus U = U^{c}$, the choice of $U^{c}$ ensures that
each edge
in $\Gamma$ has a non-empty intersection with a connected component in
$T\setminus U$. It is left to show that no connected component in
$U^{c}$ can intersect two edges of $\Gamma$.

Assume towards contradiction that there is a connected component in
$U^{c}$ that intersects more than one edges of $\Gamma$. Then there
must be two vertices $u_i,u_j$ chosen into $U^{c}$ that are
connected by an edge of $E(T)\setminus \Gamma$. This means that in
$G_T$ there is a set of vertices of size at least $h$ corresponding to
$u_j$ and another set of size $h$ corresponding to $u_i$, where 
any choice of
two vertices, one from each set, can be completed into a copy of $T$.
Thus all the vertices in the first set must be connected to all
those
in the second. But in this case $G_T$ contains a complete bipartite
graph $K_{h,h}$ and hence 
contains a copy of $H$, contradicting the assumptions.

We conclude that in $T\setminus U$, each connected component
intersects a single edge of $\Gamma$, and that each edge in $\Gamma$
intersects a single connected component in $T\setminus U$, as needed.
Note also that by the discussion above 
$T\setminus U$ is a vertex disjoint union of edges and single
vertices.

$G_T$ contains a $(U,h)$ blow-up of $T$, 
as it contains $h$ copies of each
vertex in $U^{c}$ and they are connected in $G$ as needed to
form the required blow-up.

$2\Rightarrow 3$ This is obvious.

$3\Rightarrow 1$ Assume there is a minimum
edge cover $\Gamma$ of $T$ and a
$(U,h)$ blow up of $T$ that does not contain a copy of $H$ with $U$ being a
$U(\Gamma)$ set. Any $H$ free graph has at most $hn$ edges, thus by
\cite{Al} the number of copies of $T$ in such a graph is at most
$O(n^{\alpha(T)})$, providing the required upper bound.

For the lower bound let $G_T$ be a $(U,\frac{n-|U|}{t-|U|})$ blow
up of $T$. We claim that if 
the $(U,h)$ blow up of $T$ 
does not  contain a copy of $H$, then $G_T$ does not
contain one either. Indeed, as in the proof of Theorem
\ref{thm:2trees},
if we assume that the $(U,\frac{n-|U|}{t-|U|})$ blow-up 
does contain a copy of
$H$ then this copy uses at most $h$ vertices 
and hence must be
contained in the $(U,h)$ blow-up as well, contradicting the
assumption.

$T\setminus U$ has $\alpha(T)=|\Gamma|$ 
connected components as $U$ is a $U(\Gamma)$-set.
There are $\Theta(n)$
choices for each connected component in $T\setminus U$, and each choice
produces a copy of $T$ in $G_T$. 
Thus we have $\Theta(n^{\alpha(T)})$ copies
and $ex(n,T,H)=\Theta(n^{\alpha(T)})$.
\end{proof}

Theorem \ref{thm:TreeAndBi} provides a characterization of the
pairs $(T,H)$ of a bipartite graph $T$ and a tree $H$ for which
$ex(n,T,H)=\Theta(n^{\alpha(T)})$. This characterization yields the
following result about the complexity of the corresponding
algorithmic problem.

\begin{theorem} 
The  problem of deciding, for a
given input consisting of a bipartite graph $T$ and a tree $H$, if
$ex(n,T,H)=\Theta(n^{\alpha(T)})$ is $co-NP$-hard.
\end{theorem}

\begin{proof} 
Let $G$ be a graph on $m$ vertices with minimum degree at
least $4$. Let $T$ be the bipartite graph $(A\cup B,E)$ with $A$ being the
set of vertices of $G$, $V(G)$, and $B$ being the set of
edges of $G$, $E(G)$. The
edges of the bipartite graph are as follows, a couple $\{v,e\}$ is an
edge in $T$ if $v\in e$ in $G$. Let $H$ be a path of length $2m+1$.

\begin{claim}
\label{c95}
$ex(n,T,H)=\Theta(n^{\alpha(T)})$ if and only if the graph 
$G$ does not contain a Hamilton path.  
\end{claim}

Before proving this claim note that
it is well known that the
problem of deciding if an input  graph $G$ contains a Hamilton path
is $NP$-complete (see, e.g., \cite{GJ}). It is not difficult to
show that for any
fixed $\delta$ (and in particular for $\delta=4$) this problem
remains $NP$-hard even when restricted to input graphs of minimum
degree at least $\delta$. To see this, consider a graph 
$F$ with minimum degree $d$. Let $F'$ be the graph obtained
from $F$ by adding to it a set of $(d+2)$ new vertices that
form a clique, and by joining one of the vertices of this clique 
to all vertices of $G$. It is easy to check that $F$ has a Hamilton 
path if and only if $F'$ has such a path, and that the minimum
degree of $F'$ is $d+1$. Repeating this argument we conclude
that indeed the Hamilton path problem  remains $NP$-hard when
restricted to input graphs of minimum degree at least $\delta$.
It thus suffices to prove Claim \ref{c95} in order to establish
the theorem.

To prove this claim, we choose a specific minimum 
edge cover $\Gamma$ in $T$, and show that for this edge cover
checking if the second condition
in Theorem \ref{thm:TreeAndBi} holds is equivalent to deciding if
the graph $G$ contains a Hamilton path.

Since the degree of each vertex in $A$ is at least $4$ and the
degree of each vertex of $B$ is $2$, it follows from Hall's Theorem
that the graph $T$ contains $|A|$ vertex disjoint stars, one
centered at each vertex of $A$, with each star having two leaves.
We can now complete these arbitrarily into a minimum edge cover, by
connecting each vertex of $B$ that does not lie in these stars to
an arbitrary neighbor in $A$. This provides a minimum edge cover
$\Gamma$ in which every connected component is a star with at
least two leaves.

The only possible $U(\Gamma)$-set for this $\Gamma$ is $A$. The
$(A,2m+2)$ blow-up of $T$ 
contains a copy of $H$, which is a path of length $2m+1$, 
if and only if  $G$ contains a Hamilton path, as such a path must alternate
between $A$ and $B$, and can visit each vertex in $A$ only once.

\end{proof} %

\section{Concluding remarks and open problems}

\begin{itemize} 
\item 
In Section \ref{sec:HisTree} we have shown several
cases in which when 
$H$ is a tree, $ex(n,T,H)=\Theta(n^k)$ where $k$ is an
integer. We believe that this phenomenon is more general, and that 
if $H$ is a tree then for any graph $T$,
$ex(n,T,H)=\Theta(n^{k(T,H)})$ for some integer
$k=k(T,H)$. 
\item One of the cases we focused on is $ex(n,K_3,H)$.  Even in
this special case there are many difficult problems that remain open.
One such problem that received a considerable amount of attention is the
case that $H$ is the $2$-book, that is, two triangles sharing an edge.
This is equivalent to the problem of obtaining tight bounds for the
triangle removal lemma, which is wide open despite the fact we know that
here $n^{2-o(1)} \leq ex(n,K_3,H) \leq o(n^2)$ and despite some recent
progress in \cite{Fo}.  

The determination of
$ex(n,K_3,H)$ is  complicated in many cases, 
and we do not even know its correct
order of magnitude for some simple graphs like odd cycles.
In this specific case, however, it may be that the lower bound in
(\ref{e12}) and the upper bound in Proposition \ref{p13} differ only by
a constant factor, as it may be true that the functions  $ex(m,C_{2k})$
and $ex_{bip}(m,C_4, C_6, \ldots ,C_{2k})$ differ only by a constant
factor. The problem of determining the correct order of magnitude of
$ex(n,K_3,K_{s,s,s})$ also seems complicated, the method in \cite{Ni}
yields some upper estimates. 
\item 
If $G$ contains no copy of some fixed
tree $H$ on $t+1$ vertices, then the minimum degree of $G$ is smaller
than $t$. Thus there is a vertex $v$ contained in at most ${{t-1}
\choose m}$ copies of $K_m$, and we can omit it and apply induction to
conclude that in this case $ex(n,K_m,H) < t^m n/m!$. It may be that for
any such tree the $H$-free graph $G$ on $n$ vertices maximizing the
number of 
copies of $K_m$ is a disjoint union of cliques all of which besides
possibly one are of  size $t$. As mentioned in the beginning of Section
4 this is open even for  $H=K_{1,t}$. 
\item 
As done for the classical
Tur\'an problem of studying the function $ex(n,{\cal H})$ for finite or
infinite classes ${\cal H}$ of graphs, the natural extension $ex(n,T,
{\cal H})$, which is the maximum number of copies of $T$ in a graph on
$n$ vertices containing no member of $\HH$, can also be studied. Unlike
the case $T=K_2$, there are simple examples here in which
$\HH=\{H_1,H_2\}$ contains only two graphs, and $ex(n,T,\HH)$ is much
smaller than each of the quantities $ex(n,T,H_1)$ and $ex(n,T,H_2)$. It
will be interesting to further explore this behavior. \item Another
variant of the problem considered here is that of trying to maximize the
number of  copies of $T$ in an $n$-vertex graph, given the number of
copies of $H$ in it. The case $H=K_2$ has been studied before, see
\cite{Al}, \cite{JOR}, but the general case seems far more complicated.
\item One of the exciting developments in Extremal Combinatorics in
recent years has been the  study of sparse random analogs of classical
combinatorial results, like Tur\'an's Theorem, Ramsey's Theorem, and
more. This was initiated in \cite{FR} and studied in several papers
including \cite{RR} and \cite{KLR}, culminating in the papers \cite{CG}
and \cite{Sc}.  See also \cite{BMS} and \cite{ST} for a more recent
effective approach  for investigating these problems. The natural sparse
random version of the  basic problem considered here is the study of the
following function. For two graphs $H$ and $T$ with no isolated vertices
and for a real $p \in [0,1]$, let $ex(n,T,H,p)$ be the expected value of
the maximum number of copies of the graph $T$ in an $H$-free subgraph of
the random graph $G(n,p)$. Thus $ex(n,T,H,1)$ is the function
$ex(n,T,H)$ studied here. The  behavior of $ex(n,T,H,p)$ for $T=K_2$ is
quite well understood in many cases, by the results in the papers
mentioned above, and it seems interesting to investigate the behavior of
the more general function. \end{itemize} \vspace{0.2cm}

\noindent {\bf Acknowledgment:} We thank Andrey Kupavskii and Benny
Sudakov for helpful discussions, and Dhruv Mubayi for informing us about
\cite{KMV}.

 \end{document}